\documentclass[9pt,a4paper,twoside,fleqn,reqno]{amsart}
\usepackage{,enumerate}
\theoremstyle{plain}
\newtheorem{thm}{Theorem}[section]

\newtheorem{lem}[thm]{Lemma}
\theoremstyle{definition}

\numberwithin{equation}{section}

\begin{document}
\mathindent1em
\jot4pt

\title[$\ldots$]
{Normal Curvature of Pseudo-umbilical \\Submanifolds in a Sphere}
\author{Majid Ali Choudhary}
\address{Majid Ali Choudhary, Department of Mathematics, Zakir Husain Delhi College (E), Delhi, India.}
\email{majid\_alichoudhary@yahoo.co.in}

\date{}
\begin{abstract}\baselineskip12pt

Let $M^n$ be a compact pseudo-umbilical submanifold  of the unit sphere $S^{n+p}$.
In the present note, it is shown that if the normal curvature $K^\bot$, scalar curvature $S$ and square of the length of second fundamental form $\sigma$ satisfy
$$K^\bot\leq\sigma \text{       ,      } S>(n-1)^2 $$
then $M^n$ is totally geodesic.

\end{abstract}

\maketitle
\thispagestyle{empty}
\baselineskip15pt
\section{Introduction}

Let $S^{n+p}$ be an $(n+p)$ dimensional unit sphere and $M$ be a compact $n$-dimensional
submanifold isometrically immersed in $S^{n+p}$.\ Let $h$ be the second fundamental form of the immersion
and $H$ be the mean curvature vector.\ Denote by $\langle\cdot,\cdot\rangle$ the scalar product of $S^{n+p}$.\ If
there exists a function $\lambda$ on $M$ such that
\begin{align}
\hskip11em \langle h(X,Y),H\rangle =\lambda \langle X,Y\rangle
\end{align}
for any tangent vectors $X$, $Y$ on $M$, then $M$ is called a Pseudo-umbilical submanifold of $S^{n+p}$ (\cite{2},\cite{10}).\
It is clear that $\lambda\geq 0$.\ If the mean curvature vector $H=0$ identically, then $M$ is called a minimal
submanifold of $S^{n+p}$.\ Every minimal submanifold of $S^{n+p}$ is itself a pseudo-umbilical submanifold.

Let $M$ be a compact $n$-dimensional Pseudo-umbilical submanifold of the unit sphere $S^{n+p}$ with normal bundle $\nu$. We denote by $R^\bot$ the curvature tensor field corresponding to the normal connection $\nabla^\perp$ in the normal bundle $\nu$ of $M^n$, and define $K^\bot :M \rightarrow R$ by
\begin{align}
\hskip9em K^\bot = \sum\limits_{i,j,\alpha,\beta} [R^\bot (e_i,e_j,N_{\alpha},N_\beta)]^2
\end{align}
where $\{e_1,...,e_n\}$ is a local orthonormal frame on $M^n$ and $\{N_1,...,N_p\}$ is a local field of orthonormal normals. we call the function $K^\bot$ the normal curvature of $M^n$.

There are several results for compact minimal submanifolds in a unit sphere. Simon \cite{si}, in his famous paper gave a pinching theorem, which led to an intrinsic rigidity result. Later on, Simon's work was improved by Sakaki \cite{sa} for arbitrary codimension. Shen \cite{sh} further improved the result of Sakaki but only for dimension $n=3$. Deshmukh \cite{de} partially generalized the result of Shen and proved $M^n$ to be totally geodesic with the impositions of certain conditions on normal curvature $K^\bot$, scalar curvature $S$ and square of the length of second fundamental form $\sigma$. On the other hand, Choudhary \cite{10} studied pseudo-umbilical hypersurfaces in the unit sphere. Inspired by all the above developments, we study Pseudo-umbilical submanifolds in a unit sphere and prove the following result.
\begin{thm}
Let $M$ be compact pseudo-umbilical submanifold of the unit sphere $S^{n+p}$ with normal bundle $\nu$. If normal curvature $K^\bot$, scalar curvature $S$ and square of the length of second fundamental form $\sigma$ satisfy
$$K^\bot\leq\sigma        \text{       ,      }        S>(n-1)^2 $$
then $M^n$ is totally geodesic.
\end{thm}

Infact, we have worked to investigate the normal curvature of Pseudo-umbilical submanifolds in a sphere and tried to generalize the results due to Deshmukh \cite{de}.

\section{Preliminaries}
Let $M$ be an n-dimensional pseudo-umbilical submanifold of the unit sphere $S^{n+p}$ with normal bundle $\nu$.\ We denote
by $g$ the Riemannian metric
on $S^{n+p}$ as well as the induced metric on $M$.\ Let $\nabla$ be the Riemannian connection
and $A$ be the shape operator
of the submanifold $M$. Then the second fundamental form $h$ of $M^n$ satisfies
\begin{align}
\hskip4em(\nabla h)(X,Y,Z)=(\nabla h)(Y,Z,X)=(\nabla h)(Z,X,Y)
\end{align}
for $X,Y,Z\in\mathcal{X}(M)$, where $\mathcal{X}(M)$ is the Lie algebra of smooth vector fields on $M$ and $(\nabla h)(X,Y,Z)$ is defined by
\begin{align*}
\hskip4em (\nabla h)(X,Y,Z)&=\nabla_X ^\perp h(Y,Z)-h(\nabla_X Y,Z)-h(Y,\nabla_X Z),
\end{align*}
where $\nabla^\perp$ is the connection defined in $\nu$. The second covariant derivative $(\nabla^2 h)(X,Y,Z,W)$ of the second fundamental form is given by
\begin{align*}
\hskip4em (\nabla^2 h)(X,Y,Z,W)&=\nabla_X ^\perp (\nabla_X) (Y,Z,W)-(\nabla h)(\nabla_X Y,Z,W)  \\
     &\quad-(\nabla h)(Y, \nabla_X Z,W)-(\nabla h)(Y,Z,\nabla_X W)
\end{align*}
for $X,Y,Z,W\in\mathcal{X}(M)$. The Ricci identity is given by
\begin{align}
(\nabla^2 h)(X,Y,Z,W)-(\nabla^2 h)(Y,X,Z,W)& = R^\perp (X,Y)h(Z,W)-h(R(X,Y)Z,W)\nonumber\\
&\quad-h(Z,R(X,Y)W)
\end{align}
for $X,Y,Z,W\in\mathcal{X}(M)$, where $R^\perp$ and $R$ are the curvature tensors of the connections $\nabla^\perp$ and $\nabla$ respectively. Since $M^n$ is a Pseudo-umbilical submanifold, then for a local orthonormal frame $\{e_1,...,e_n\}$ of $M^n$ we have
\begin{align}
\hskip9em \sum\limits_{i=1}^n h(e_i,e_i)=nH
\end{align}
We define the symmetric operator $R^*$ by using the Ricci tensor Ric in the following way
$Ric(X,Y)=g(R^* (X),Y)$,                   for $X,Y\in\mathcal{X}(M)$.
Then the Gauss equation gives
\begin{align}
A_{h(Y,Z)} X&=R(X,Y)Z+A_{h(X,Z)} Y -g(Y,Z)X+g(X,Z)Y\\
R^* (X)&=(n-1)X -\sum\limits_{i=1}^n A_{h(e_i,X)} e_i +\sum\limits_{i=1}^n A_{h(e_i,e_i)} X \\
S&=n(n-1)+n^2 H^2-\|A_\alpha\|^2
\end{align}
where $A_N ,N \in\nu$ is the Weingarten map with respect to the normal $N$, satisfying
$$g(A_N X,Y)=g((h(X,Y),N).$$
We define
\begin{align}
\hskip7em\sigma &= \sum\limits_{i,j} \|h(e_i,e_j)\|^2,\nonumber\\
\hskip7em\|A_h\|^2&=\sum\limits_{i,j,k} \|A_{h(e_i,e_j)} e_k\|^2, \\
\hskip7em\|\nabla_h\|^2&=\sum\limits_{i,j,k} \|(\nabla h) (e_i,e_j,e_k)\|^2.\nonumber
\end{align}

Now, we state the following lemmas which are required for the proof of our main theorem.

\def\Tr{\mathop{\rm Tr}}
\begin{lem}
Let $M^n$ be an immersed pseudo-umbilical submanifold of the unit sphere $S^{n+p}$, then for a local orthonormal frame $\{e_1,\ldots,e_n\}$, we have 
$\sum\limits_{i,j,k} R(e_k , e_i ; e_j , A_{h(e_i , e_j)} e_k) = n^2 H^2 - \sigma + {\|A_h\|}^2 + \frac{1}{2} K^{\perp} - \sum\limits_{i,j,\alpha,\beta} g(A_\alpha e_i , A_\beta e_j )^2$,                                                                                            where $A_\alpha \equiv A_{N_\alpha}$ and $\{N_1,\ldots,N_p\}$ is a local field of orthonormal normals.
\end{lem}

\begin{proof}
Using the Gauss equation, we have
\begin{align}
R(e_k , e_i ; e_j , A_{h(e_i , e_j)} e_k) &=  g(e_i ,e_j ) g(e_k ,A_{h(e_i , e_j )} e_k ) - g(e_k ,e_j ) g(e_i ,A_{h(e_i , e_j )} e_k ) \nonumber\\
&+ g(A_{h(e_i , e_j )} e_k , A_{h(e_i , e_j )} e_k ) - g(A_{h(e_k , e_j )} e_i ,A_{h(e_i , e_j )} e_k ) \nonumber\\
&=\delta_{ij} g(h(e_k , e_k ), h(e_i , e_j )) -\delta_{kj} g(h(e_i , e_j ), h(e_i , e_k )) \\
&+g(A_{h(e_i , e_j )} e_k , A_{h(e_i , e_j )} e_k )- g(h(e_k , e_j ) , h(e_i ,A_{h(e_i , e_j )} e_k ))\nonumber
\end{align}
Using the fact that, $A_{h(e_i , e_j )} e_k = \sum\limits_{\alpha} g(A_\alpha e_i , e_j )A_\alpha e_k$ , we obtain
\hskip4em \begin{align}g(h(e_k , e_j ) , h(e_i ,A_{h(e_i , e_j )} e_k )) = \sum\limits_{i,j,\alpha,\beta} g(A_\alpha A_\beta e_i , e_j)g(A_\beta A_\alpha e_i , e_j)
\end{align}
We also have
\begin{align}
\hskip8em g(A_{h(e_i , e_j )} e_k , A_{h(e_i , e_j )} e_k ) = \|A_h\|^2
\end{align}
On the other hand, using the Ricci equation
$R^{\perp} (X,Y;N_1 , N_2) = g([A_{N_1} , A_{N_2}](X), Y)$,   for $X,Y\in\mathcal{X}(M), N_1 , N_2 \in\nu,$  we have
\begin{align}
K^\perp &= \sum\limits_{i,j,\alpha,\beta} [R^\perp (e_i , e_j ; N_\alpha , N_\beta )]^2 \nonumber\\
&= \sum\limits_{i,j,\alpha,\beta} [g(A_\alpha A_\beta e_i , e_j) -  g(A_\beta A_\alpha e_i , e_j)]^2 \nonumber\\
&= 2 \sum\limits_{i,j,\alpha,\beta}  g(A_\alpha e_i , A_\beta e_j)^2 -  2 \sum\limits_{i,j,\alpha,\beta} g(A_\alpha A_\beta e_i , e_j)  g(A_\beta A_\alpha e_i , e_j)\nonumber\\
\Rightarrow \frac{1}{2} K^\perp - &\sum\limits_{i,j,\alpha,\beta}  g(A_\alpha e_i , A_\beta e_j)^2 = - \sum\limits_{i,j,\alpha,\beta} g(A_\alpha A_\beta e_i , e_j)  g(A_\beta A_\alpha e_i , e_j)
\end{align}
where we have used $\sum\limits_{i,j,\alpha,\beta}  g(A_\alpha e_j , A_\beta e_i)^2 = \sum\limits_{i,j,\alpha,\beta}  g(A_\beta e_j , A_\alpha e_i)^2$ which follows from the symmetry of $A_\alpha$ and $A_\beta$.

Finally, in the light of equations (2.9), (2.10) and (2.11) and using the fact that $M^n$ is a Pseudo-umbilical submanifold, equation (2.8) reduces to \\
$\sum\limits_{i,j,k} R(e_k , e_i ; e_j , A_{h(e_i , e_j)} e_k) = n^2 H^2 - \sigma + {\|A_h\|}^2 + \frac{1}{2} K^{\perp} - \sum\limits_{i,j,\alpha,\beta} g(A_\alpha e_i , A_\beta e_j )^2     $\\
and this completes the proof of the lemma.
\end{proof}

Next, we prove

\mathindent4em
\begin{lem}
Let $M^n$ be a pseudo-umbilical submanifold of the unit sphere $S^{n+p}$, then for a local orthonormal frame $\{e_1,\ldots,e_n\}$, we have
\begin{align*}
\hskip-4em\sum\limits_{i,j,k}  R^\perp (e_k, e_i ; h(e_k , e_j) , h(e_i , e_j )) &= \|A_h\|^2 + n^2 H^2 -n \sigma    \nonumber\\
&+ \sum\limits_{i,j,k} [ R(e_i ,e_k , e_j , A_{h(e_i,e_j)} e_k )  - R(e_k ,e_i , e_k , A_{h(e_i,e_j)} e_i ) ]  \nonumber\\
&-\sum\limits_{i,j,k} g( A_{h(e_k,e_k)} e_j , A_{h(e_i,e_j)} e_i ).
\end{align*}
\end{lem}

\begin{proof}
Let $M^n$ be a pseudo-umbilical submanifold of the unit sphere $S^{n+p}$. Then, using (2.4) in the Ricci equation, and in the light of (2.7), we obtain
\begin{align}
\sum\limits_{i,j,k}  R^\perp (e_k, e_i ; h(e_k , e_j) , h(e_i , e_j )) &= \|A_h\|^2 + \sum\limits_{i,j,k} [ R(e_i , e_k ; e_j , A_{h(e_i ,e_j )} e_k) \nonumber\\
&- g(A_{h(e_i ,e_j )} e_j ,e_i) + g(h(e_j , e_j ), h(e_k , e_k )) \nonumber\\
&- g(A_{h(e_k ,e_j )} e_k , A_{h(e_i ,e_j )} e_i) ]
\end{align}
Now, taking into consideration equation (2.5), we have
\begin{align*}
R^* (e_j)=(n-1)e_j -\sum\limits_{k}^n A_{h(e_j,e_k)} e_k +\sum\limits_{k}^n A_{h(e_k,e_k)} e_j
\end{align*}
Taking inner product with $A_{h(e_i,e_j)} e_i$ in the above equation, we obtain  
\begin{align*}
\hskip-3em\sum\limits_{i,j} g(R^* (e_j) , A_{h(e_i,e_j)} & e_i ) =(n-1) \sum\limits_{i,j} g(A_{h(e_i,e_j)} e_i , e_j) \\
&+\sum\limits_{i,j,k} g( A_{h(e_k,e_k)} e_j , A_{h(e_i,e_j)} e_i ) -\sum\limits_{i,j,k} g(A_{h(e_j,e_k)} e_k , A_{h(e_i,e_j)} e_i )
\end{align*}
So, in the light of equation (2.7), above equation gives 
\begin{align}
\hskip-4em\sum\limits_{i,j,k} g(A_{h(e_j,e_k)} e_k ,  A_{h(e_i,e_j)} e_i ) &= (n-1) \sum\limits_{i,j} g(h(e_i,e_j) , h(e_i , e_j)) \nonumber\\
& +\sum\limits_{i,j,k} g( A_{h(e_k,e_k)} e_j , A_{h(e_i,e_j)} e_i ) - \sum\limits_{i,j} g(R^* (e_j) , A_{h(e_i,e_j)} e_i )\nonumber\\
& = (n-1)\sigma + +\sum\limits_{i,j,k} g( A_{h(e_k,e_k)} e_j , A_{h(e_i,e_j)} e_i ) \nonumber\\
&- \sum\limits_{i,j} Ric(e_j , A_{h(e_i,e_j)} e_i )\nonumber\\
&=(n-1)\sigma  +\sum\limits_{i,j,k} g( A_{h(e_k,e_k)} e_j , A_{h(e_i,e_j)} e_i ) \nonumber\\
&+ \sum\limits_{i,j,k} R(e_k ,e_i , e_k , A_{h(e_i,e_j)} e_i ) 
\end{align}
where, we have used the relationship between operator $R^*$ and Ricci tensor Ric. Thus, using using (2.13) in (2.14), we have the required result.
\end{proof}

\section{Main Results}

\begin{proof}[Proof of Theorem~{\rm1.1}]
Let $M^n$ be a compact Pseudo-umbilical submanifold of $S^{n+p}$ satisfying the hypothesis of the theorem. We define a function \\
     $F : M \rightarrow R      \hskip1em \text{by}
     \hskip1em   F=\frac{1}{2} \sigma.$ \\
Then, the Laplacian $\triangle F$ of the function $F$ can be computed as
\begin{align*}
\triangle F = \sum\limits_{i,j,k} \|(\nabla h)(e_i , e_j , e_k) \|^2 + \sum\limits_{i,j,k} g((\nabla^2 h)(e_k ,e_k, e_i , e_j ),h(e_i , e_j))
\end{align*}
In view of equation (2.1) and the Ricci identity (2.2), above equation gives
\begin{align*}
\triangle F &= \sum\limits_{i,j,k} \|(\nabla h)(e_i , e_j , e_k) \|^2 + \sum\limits_{i,j,k} [ R^\perp (e_k, e_i ; h(e_k , e_j) , h(e_i , e_j )) \nonumber\\
 &-R(e_k , e_i ; e_k , A_{h(e_j ,e_i )} e_j) - R(e_k , e_i ; e_j , A_{h(e_i ,e_j )} e_k)] + \mu
\end{align*}
where, we have assumed $\mu = \sum\limits_{i,j,k} g((\nabla^2 h)(e_j , e_i , e_k ,e_k ) , h( e_i , e_j ))$. 
Now, taking into account equation (2.7), one can write above equation as follows
\begin{align}
\triangle F &= \|\nabla h \|^2 + \sum\limits_{i,j,k} [ R^\perp (e_k, e_i ; h(e_k , e_j) , h(e_i , e_j )) \nonumber\\
 &-R(e_k , e_i ; e_k , A_{h(e_j ,e_i )} e_j) - R(e_k , e_i ; e_j , A_{h(e_i ,e_j )} e_k)] + \mu
\end{align}
On the other hand,
\begin{align*}
\hskip2em \sum\limits_{i,j,k} R(e_k , e_i ; e_k , A_{h(e_i ,e_j )} e_j) 
&= - \sum\limits_{i,j} Ric(e_i , A_{h(e_i , e_j )} e_j)     \\         
&= - \sum\limits_{i,j} g(R^* e_i , A_{h(e_i , e_j )} e_j), 
\end{align*}
since $A_{h(e_i ,e_j)} e_k = \sum\limits_{\alpha} g(A_\alpha e_i, e_j) A_\alpha e_k $, we obtain
\begin{align}
\sum\limits_{i,j,k} R(e_k , e_i ; e_k , A_{h(e_i ,e_j )} e_j) 
&= - \sum\limits_{i,j,\alpha} g(R^* e_i , A_\alpha e_j ) g( A_\alpha e_i, e_j) \nonumber\\
&= - \sum\limits_{i,j,\alpha} g(R^* A_\alpha e_j ,e_i) g( A_\alpha e_j, e_i) \nonumber\\
&= - \sum\limits_{i,j,\alpha} g(R^* A_\alpha e_j , A_\alpha e_j) \nonumber\\
&= - \sum\limits_{j,\alpha} Ric(A_\alpha e_j , A_\alpha e_j) \nonumber\\
&= -(n-1)\sigma + \sum\limits_{i,j,\alpha , \beta} g(A_\beta e_i , A_\alpha e_j)^2 - n^2 H^2 \|A_\alpha\|^2
\end{align}
Using (3.2) in (3.1) and taking view of lemma 2.1 and lemma 2.2, we have
\begin{align}
\triangle F &= (\|\nabla h \|^2 -\|A_h \|^2) - n^2 H^2 +n \sigma - K^{\perp} +\mu + n^2 H^2 \|A_\alpha \|^2  
\end{align}
In the light of (2.7), we also have
\begin{align}
\|A_h \|^2  &= \sum\limits_{i,j,k} \|A_{h(e_i ,e_j )} e_k\|^2 = \sum\limits_{i,j,k,\alpha} g(A_\alpha e_i , e_j)^2 \|A_\alpha e_k\|^2 \nonumber\\
&=\sum\limits_{i,j,\alpha} g(A_\alpha e_i , e_j)^2 \|A_\alpha \|^2   =  \sum\limits_{\alpha} \|A_\alpha\|^2 \|A_\alpha \|^2 \nonumber\\
&= \|A_\alpha \|^4
\end{align}
Let us suppose that $\sigma = \sum\limits_{\alpha} \|A_\alpha\|^2$. Then, from (3.3) and (3.4), we have
\begin{align*}
\hskip-3em\triangle F &= [(n-1) -\{\|A_\alpha\|^2 - n^2 H^2 \}]\|A_\alpha \|^2  + (\sigma - K^{\perp}) +\mu +  \|\nabla h\|^2 -n^2 H^2
\end{align*}
Integrating over $M^n$, we obtain
\begin{align}
\hskip-3em\int_M\{ \sum\limits_{\alpha} [(n-1) -\{\|A_\alpha\|^2 - n^2 H^2 \}]\|A_\alpha \|^2  + (\sigma - K^{\perp}) +\mu +  \|\nabla h\|^2 -n^2 H^2 \} dv = 0.
\end{align}
Now, as per the hypothesis of the theorem 
$ S>(n-1)^2$, it follows that
$ n(n-1) +n^2 H^2 - \sum\limits_{\alpha} \|A_\alpha \|^2 >(n-1)^2$.
That is, 
$ n^2 H^2 - \sum\limits_{\alpha} \|A_\alpha \|^2 <(n-1)$, and hence, $ n^2 H^2 - \|A_\alpha \|^2 >(n-1)$, and that $K^\perp \leq \sigma$. Thus, in order for (3.5) to hold, we must have $\|A_\alpha\|=0$, that is $M^n$ is totally geodesic and this proves the theorem.
\end{proof}

\mbox{}
\end{document}